\documentclass[12pt]{article}
\usepackage{amssymb,amsmath}
\usepackage{latexsym,bm}
\usepackage{dsfont}

\usepackage{graphicx,amssymb,mathrsfs,amsmath,latexsym,amsfonts,amsthm}
\usepackage{amsmath,amssymb,amsthm}
\usepackage{latexsym}
\usepackage{amssymb}
\usepackage{stmaryrd}
 \usepackage{float}
\usepackage{psfrag}
\usepackage{xcolor}
\usepackage{enumerate}
\usepackage{manfnt,mathrsfs}
\usepackage{epstopdf}

\makeatletter 
\@addtoreset{equation}{section}
\makeatother
\newtheorem{theorem}{Theorem}

\newtheorem{lemma}[theorem]{Lemma}

\newtheorem{problem}[theorem]{Problem}
\def\adots{\mathinner{\mkern2mu\raise0pt\hbox{.}  
\mkern2mu\raise4pt\hbox{.}\mkern1mu
\raise7pt\vbox{\kern7pt\hbox{.}}\mkern1mu}}

\newcommand{\A}{\mathcal{A}}

\newcommand{\V}{\mathcal{V}}

\def\tu{{\rm Tur{\'a}n\,}}
\date{ }
\begin{document}
\title{ 0-1 matrices with zero trace whose squares are 0-1 matrices}
\author{ Zejun Huang,\thanks{Institute of Mathmatics, Hunan University, Changsha  410082, P.R. China.  (mathzejun@gmail.com) } ~~Zhenhua Lyu,\thanks{College of Mathematics and Econometrics, Hunan University, Changsha  410082, P.R. China.  (lyuzhh@outlook.com)}  } \maketitle

\begin{abstract}
In this paper, we determine the maximum number of nonzero entries in 0-1 matrices of order $n$ with zero trace whose squares are 0-1 matrices when  $n\ge 8$. The extremal matrices attaining this maximum number are also characterized.
\end{abstract}


\section{Introduction}

Denote by $M_{n}\{0,1\}$ the set of 0-1 matrices of order $n$.  In 2007, Zhan \cite{ZH} proposed the following problem.
\begin{problem}\label{p1}
Given two integers $n$ and $k$, what is the maximum number of nonzero entries in a matrix $A\in M_n\{0,1\}$ such that $A^k\in M_n\{0,1\}$, and what are the extremal matrices attaining this maximum number?
\end{problem}
Wu \cite{WU} solved the case  $k=2$. Huang and Zhan \cite{HZ1} solved the case $k\ge n-1$ and they attained the maximum number for the case $k=n-2$ and $n-3$.  The authors of \cite{HLQ} solved the case $k\ge 5$ and they attained the maximum number for the case $k=4$. In this paper, we consider the following related problem.
\begin{problem}\label{p2}
	Given   integers $n$ and $k$, determine the maximum number of nonzero entries in a matrix $A \in M_n\{0,1\}$  such that $$tr(A)=0 {\rm ~and~}  A^k\in M_{n}\{0,1\}.$$ Characterize the extremal matrices that attain this maximum number.
\end{problem}
We solve the case $k=2$ for Problem \ref{p2} in this paper. Our  approach  is transferring the problem to an equivalent problem on digraphs and applying detail analysis on the structures of certain digraphs.  We need the following definitions and notations.

 We abbreviate directed
	walks, directed paths and directed cycles as walks, paths and cycles, respectively. The {\it length} of a walk, path or cycle is its number of arcs. The number of vertices in a digraph is called its {\it order} and the number of arcs its {\it size}. Let $u,w$ be two vertices. The notation $(u,w)$ or $u\rightarrow w$ means there exists an arc from $u$ to $w$;  $u\nrightarrow w$ means there exists no arc from $u$ to $w$;  $u\leftrightarrow w$ means both $u \rightarrow w$ and $w\rightarrow u$. If there exists an arc from $u$ to $w$, then $u$ is a {\it predecessor} of $w$, and $w$ is a {\it successor} of $u$.

 Let $D=(\mathcal{V},\mathcal{A})$ be a digraph with vertex set $\mathcal{V}=\{v_1,v_2,\ldots,v_n\}$ and arc set $\mathcal{A}$. Its {\it adjacency matrix} $A_D=(a_{ij})$ is defined by
\begin{equation}\label{eqh1}
a_{ij}=\left\{\begin{array}{ll}
                 1,& (v_i,v_j)\in \mathcal{A};\\
                 0,&\textrm{otherwise}.\end{array}\right.
                 \end{equation}
Conversely, given an $n\times n$ 0-1 matrix $A=(a_{ij})$, we can define its digraph $D(A)=(\mathcal{V},\mathcal{A})$ on vertices $v_1,v_2,\ldots,v_n$ by (\ref{eqh1}), whose adjacency matrix is $A$.

For a subset $X\subset \mathcal{V}$, $D(X)$ denotes the subdigraph of $D$ induced by $X$. The outdegree $d^+(u)$ is the number of arcs with tail $u$ and the indegree $d^-(u)$ is the number of arcs with head $u$. For convenience, if $X=\{x\}$ is a singleton, it will be abbreviated as $x$.

 For $S,T\subset \mathcal{V}$, denote by $\A(S,T)$ the set of arcs from $S$ to $T$. If $S=T$ we use $\A(S)$ instead of $\A(S,S)$. Let $e(S,T)=|\A(S,T)|$ and $e(D)=|\A(\V)|$. $S\rightarrow T$ means for every vertex $i\in T$ there exists exactly one vertex $j\in S$ such that $j\rightarrow i$ ; $S\nrightarrow T$ means there is no arc from $S$ to $T$. If every vertex in $S$ has exactly one successor in $T$ and each vertex in $T$ has exactly one predecessor in $S$, we say $S$ {\it matches} $T$. Note that $S$ matching $T$ indicates $|S|=|T|$.

 For $W\subset \mathcal{V}$, denote by
 $$N^+_{W}(u)=\{x\in W|(u,x)\in \mathcal{A}\},$$
  $$N^-_{W}(u)=\{x\in W|(x,u)\in \mathcal{A}\},$$
  and $$N^+_{W}(S)=\bigcup\limits_{u\in S}N^+_{W}(u),$$ where $S\subset \V$. When $W=\mathcal{V}$, we simply write $N^+(u)$, $N^-(u)$ and $N^+(S)$ respectively.

The following problem is equivalent with Problem {\ref{p1} and Problem {\ref{p2}.
\begin{problem}\label{p3}
  Given two integers $n$ and $k$, determine the maximum size of  digraphs of order $n$ avoiding two distinct directed walks of a given length $k$ with the same initial and terminal vertices. Characterize the extremal digraphs attaining this maximum size.
\end{problem}
\noindent Using the digraphs of 0-1 matrices, we see that for a  matrix $A\in M_n\{0,1 \}$, $A^k\in M_n\{0,1\}$ if and only if $D(A)$ avoids distinct directed walks of length $k$ with the same initial vertex and the same terminal vertex.   Hence, for digraphs allow loops but donot allow parallel arcs,  Problem \ref{p3}  is equivalent with Problem \ref{p1}; for strict digraphs, i.e., digraphs do not allow loops or parallel arcs, Problem \ref{p3} is equivalent with Problem \ref{p2}.

For strict digraphs, the solution to Problem \ref{p3} for the case $k\ge 5$ follows straightforward from \cite{HLQ,HZ1},  since the extremal digraphs in \cite{HLQ,HZ1} are loopless.
In this paper we consider the case $k=2$ for Problem \ref{p3} on strict digraphs.

In what follows digraphs are strict. Given a family of digraphs $\mathscr{H}$ and a digraph $D$, $D$ is $\mathscr{H}$-free if $D$ contains no member of $\mathscr{H}$ as its subgraph. Let $\mathscr{F}=\{P_{2,2},C_{2,2}\}$ where $P_{2,2}$ and $C_{2,2}$ are defined as follows.

 \begin{figure}[H]
        \centering
        \includegraphics[width=1.1in]{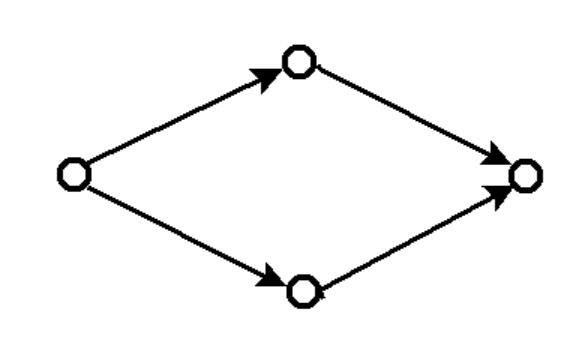}\hspace{0.8cm}
        \includegraphics[width=1.5in]{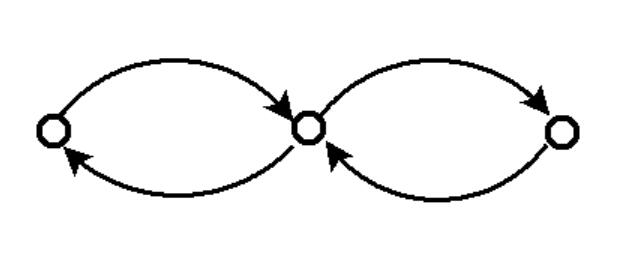}\hspace{0.8cm}\\

     $P_{2,2}$\hspace{4.1cm}$C_{2,2}$
 \end{figure} It is obvious that a digraph avoids two distinct 2-walks with the same initial and terminal vertices if and only if it is $\mathscr{F}$-free. Denote by $ex(n)$ and $EX(n)$ the maximum size of $\mathscr{F}$-free digraphs of order $n$ and the set of $\mathscr{F}$-free digraphs of order $n$ attaining the maximum size, respectively.

Note that Problem \ref{p3} is a \tu type problem, which concerns the study of  extremal graphs  that avoid given subgraphs. \tu problem is a hot topic in graph theory with a long history; see \cite{BB,BH,BS,MS,PT,PT2}.   In \cite{HL}  the authors studied a closely related \tu problem. They determined the maximum size of $P_{2,2} $-free digraphs as well as  the extremal $P_{2,2} $-free digraphs.
 \section{Main results}

If a digraph $D$ is acyclic and there is a vertex $u$ such that there is a unique directed path from $u$ to any other vertex, then we say $D$ is  an {\it arborescence} with root $u$.  If the maximum length of these paths is at most $r$, then we say $D$ is  an {\it $r$-arborescence}. Moreover, if $D$ is a 1-arborescence, we also say $D$ is an {\it out-star}. Throughout this article, we assume each arborescence has more than one vertices.
	
We will use $S(x)$, $S_{x}(y)$, $T(x)$ and $T(x,y)$ to denote the following digraphs, whose orders will be clear from the context. Note that $S(x)$ is an out-star with root $x$;
$S_{x}(y)$ is the union of a 2-cycle $x\leftrightarrow y$ and an out-star with root $y$; $T(x)$ is a  2-arborescence with root $x$; $T(x,y)$ is the union of a  2-cycle $x\leftrightarrow y$ and two 2-arborescences with roots $x$ and $y$.
\begin{figure}[H]
        \centering
        \includegraphics[width=.7in]{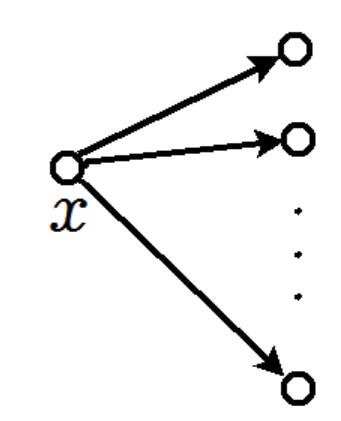}\hspace{0.8cm}
        \includegraphics[width=1in]{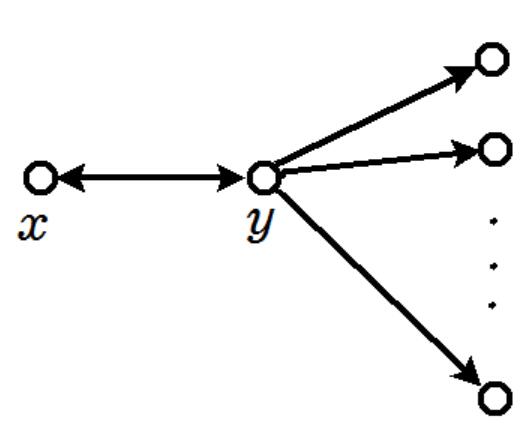}\hspace{0.8cm}
        \includegraphics[width=1in]{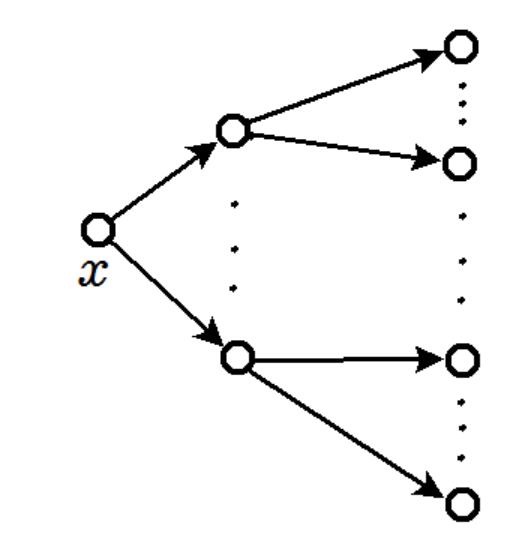}\hspace{0.8cm}
        \includegraphics[width=2in]{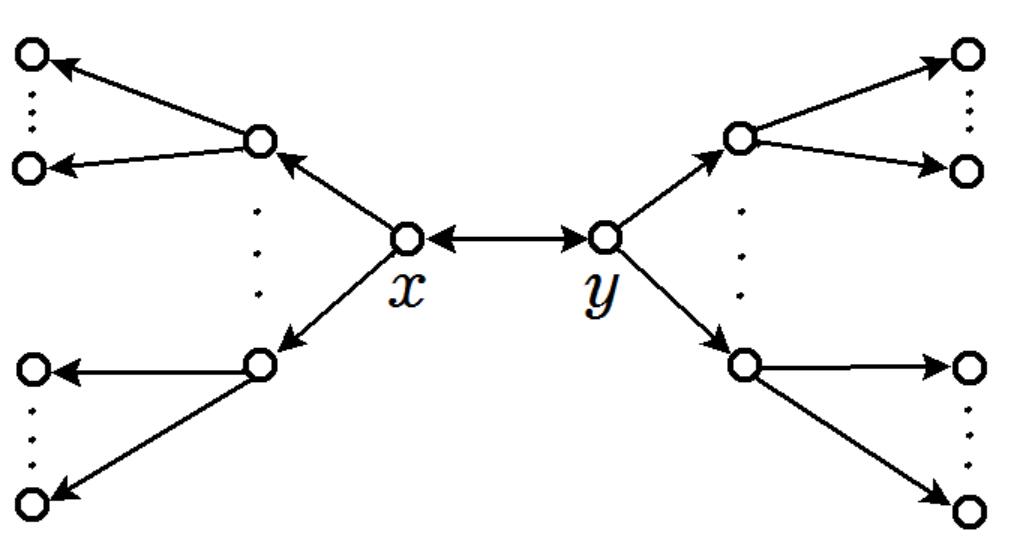}\hspace{0.8cm}\\

       $S(x)$\hspace{1.5cm}$S_{x}(y)$\hspace{3.3cm}$T(x)$\hspace{3.4cm}$T(x,y)$
 \end{figure}

Here we denote a single vertex by $C_1$. The digraph obtained by taking the union of digraphs $D$ and $H$ with disjoint vertex sets is the disjoint union, written $D+H$. In general, $mD$ is the digraph consisting of $m$ pairwise disjoint copies of $D$. Denote by $D'$ the reverse of $D$, which is obtained by reversing the directions of all arcs of $D$. The reverse of an out-star is called an {\it in-star}.

Now we present the following six classes of digraphs on $n$ vertices, where $n$ is even for $D_1$ and odd for the others. Each of these diagraphs has vertex partition $\V_1\cup \V_2$ with $|\V_1|=\lfloor\frac{n}{2}\rfloor+1$ and $|\V_2|=\lceil\frac{n}{2}\rceil -1$.

\begin{figure}[H]
		\centering
		\includegraphics[width=1.6in]{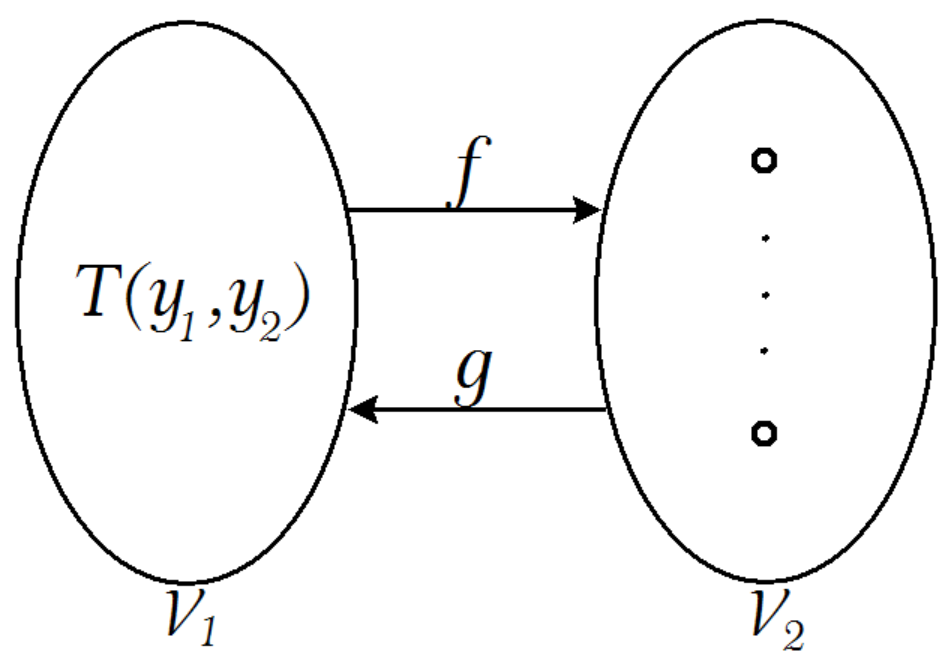}\hspace{0.1cm}
		\includegraphics[width=1.6in]{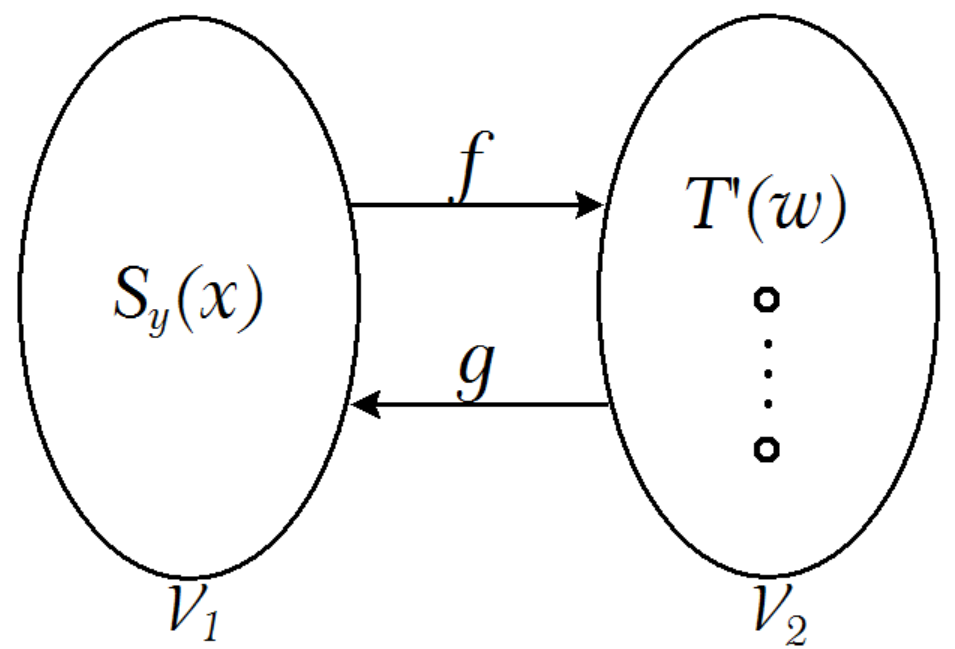}\hspace{0.1cm}
			\includegraphics[width=1.6in]{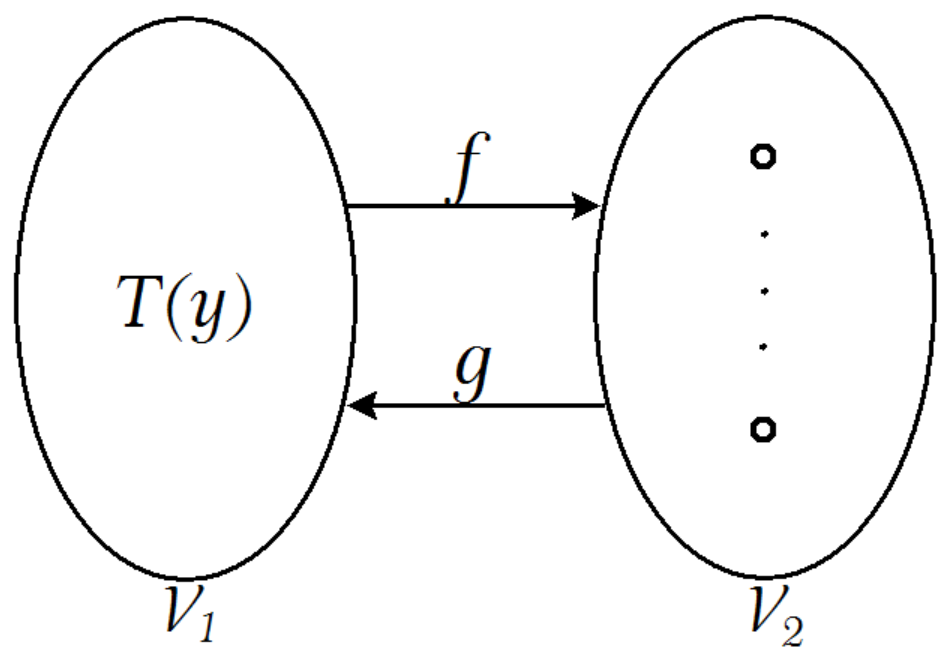} \\
		$D_1$\hspace{4cm}$D_2$\hspace{4cm}$D_3$\\
		\vspace{0.3cm}
		\includegraphics[width=1.6in]{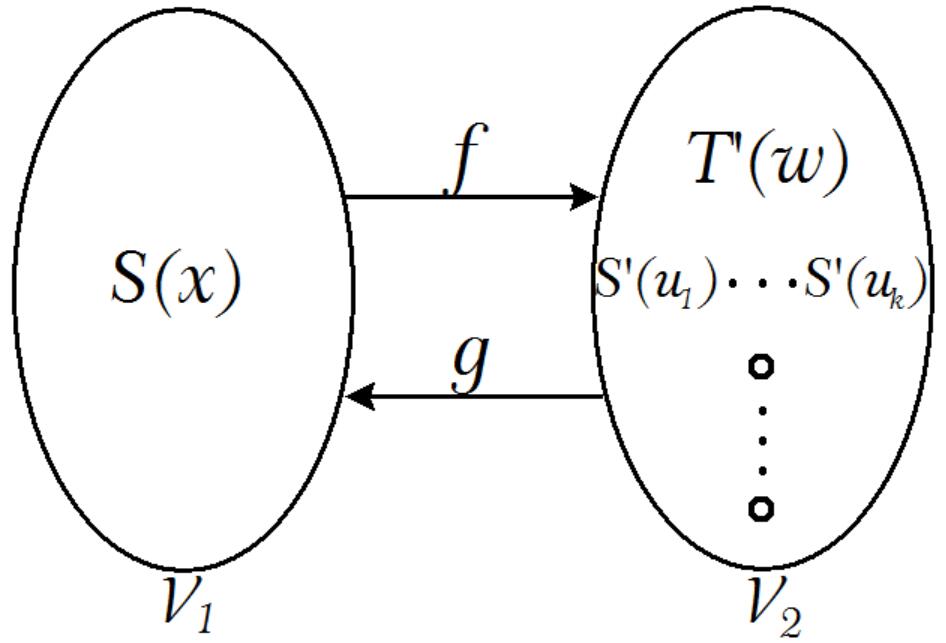}\hspace{0.1cm}
		\includegraphics[width=1.6in]{di1.png}\hspace{0.1cm}
		\includegraphics[width=1.6in]{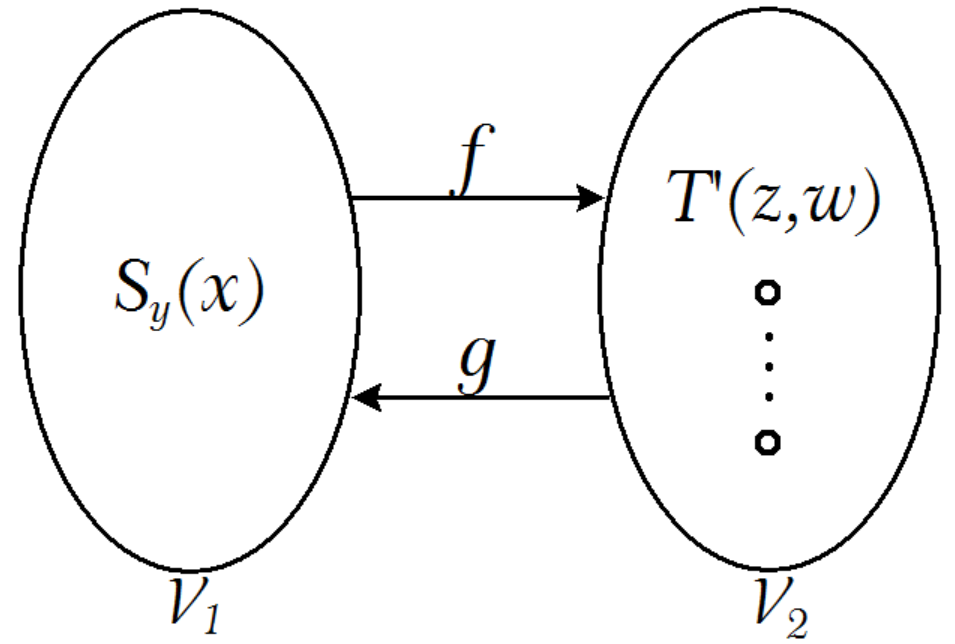} \\
		$D_4$\hspace{4cm}$D_5$\hspace{4cm}$D_6$\\
	
\end{figure}

	In $D_1$,
	$D(\V_1)=T(y_1,y_2)$;
	$D(\V_2)$ is empty;
	$f$ means $$\V_1\backslash \{y_1,y_2\}~ {\rm matches} ~\V_2;$$
	$g$ means $$u\rightarrow \V_1{\rm ~for~ all~} u\in \V_2.$$
	
	In $D_2$,
	$D(\V_1)=S_y(x)$; $D(\V_2)$ is the disjoint union of $T'(w)$ and isolated vertices, where $T'(w)$ may vanish. $\V_2$ is partitioned as $\V_2=\V_3\cup \V_4$, where $\V_3$ consists of all the isolated vertices of $\V_2$ and $|\V_3|\ge 1$;  $f$ means
	$$\V_1\backslash \{y,x\}{\rm ~matches~}\V_2\backslash \{w\} ~{\rm and~}x\rightarrow w;$$ $g$ means
	$$u\rightarrow \V_1 {\rm ~for~ all~} u\in \V_3 ~{\rm and}~ u\rightarrow \V_1\backslash  \{x\}{\rm ~for~ all~} u\in\V_4.$$
	
	In $D_3$, $D(\V_1)=T(y)$; $D(\V_2)$ is empty;  $f$ means $$\V_1\backslash \{y\}~ {\rm matches} ~\V_2;$$
	$g$ means $$u\rightarrow \V_1{\rm ~for~ all~} u\in \V_2.$$

	In $D_4$, $D(\V_1)=S(x)$;  $D(\V_2)$ is the disjoint union of $T'(w)$, in-stars and isolated vertices, and each of them may vanish. $\V_2$ is partitioned as $\V_2=\V_3\cup \V_4$, where $\V_3$ consists of all isolated vertices of $\V_2$ and the roots of $T'(w)$ and the in-stars; $f$ means
	$$\V_1\backslash \{x,y\}~ {\rm matches} ~\V_2\backslash \{w\} ~{\rm and ~}x\rightarrow w,$$  where $y$ is an arbitrary vertex in $\V_1\backslash \{x\}$; $g$ means
	$$u\rightarrow \V_1 {\rm ~for~ all~} u\in \V_3 ~{\rm and}~ u\rightarrow \V_1\backslash  \{x\}{\rm ~for~ all~} u\in\V_4.$$

	In $D_5$,  $D(\V_1)=T(y_1,y_2)$;
	$D(\V_2)$ is empty;
	$f$ means $$\V_1\backslash \{y_1,y_2\}{\rm ~matches~} \V_2\backslash \{x\},$$ where $x$ is an arbitrary vertex in $\V_2$;
	$g$ means $$u\rightarrow \V_1{\rm ~for~ all~} u\in \V_2.$$

	In $D_6$,  $D(\V_1)=S_y(x)$;  $D(\V_2)$ is the disjoint union of $T'(w,z)$ and isolated vertices. $\V_2$ is partitioned as $\V_2=\V_3\cup \V_4$, where $\V_3$ consists of all the isolated vertices of $\V_2$ and $|\V_3|\ge 1$;
	$f$ means
	$$\V_1\backslash \{y,y',x\}{\rm ~matches~} \V_2\backslash \{z,w\}~{\rm and}~x\rightarrow w,$$where $y'$ is an arbitrary vertex in $\V_1\backslash \{y,x\}$;  $g$ means $$u\rightarrow \V_1 {\rm ~for~ all~} u\in \V_3 ~{\rm and}~ u\rightarrow \V_1\backslash  \{x\}{\rm ~for~ all~} u\in\V_4.$$

We say that digraph $D $ is an isomorphism of $ H$ if there exists a bijection $f$: $\V(D)\rightarrow\V(H)$ such that $(u,v)\in \A(D)$ if and only if $(f(u),f(v))\in \A(H)$. Now we post our main result as follows.

\begin{theorem}\label{th1}

Let D be a digraph on $n$ vertices with $n\ge 8$. Then $$ex(n)=\lfloor\frac{n^2+4n}{4}\rfloor-1.$$ Moreover, $D\in EX(n)$ if and only if
\begin{itemize}
			\item[(1)] $n$ is even, and $D$ or $D'$ is an isomorphism of $D_1$;
			\item[(2)] $n$ is odd, and $D$ or $D'$ is an isomorphism of $D_i$ with $i\in \{2,\ldots,6\}$.
\end{itemize}
\end{theorem}
Equivalently, we have the following solution to Problem \ref{p2} when $k=2$.

\noindent {\bf Theorem \ref{th1}*}.
{\it Let $n\ge 8$ be an integer. Suppose $A\in M_n\{0,1\}$ such that
	$$tr(A)=0~{\rm and ~} A^2\in M_n\{0,1\}.$$ Then $A$ has  at most $ex(n)$ nonzero entries, and  $A$ has $ex(n)$  nonzero entries if and only if $D(A)$ or its reverse is an isomorphism of $D_i$ with $i\in \{1,\ldots,6\}$.}

  \section{Lemmas}
In this section we give preparatory lemmas for the proof of the main theorem.
\begin{lemma}\label{le8}
Let $n\ge 3$ be a positive integer. Then $D_i$ is $\mathscr{F}$-free  for $i=1,2,\ldots,6$. Moreover, we obtain
\begin{equation}\label{1.1}
ex(n)\ge
\begin{cases}
\frac{n^2+4n-5}{4},&\ n~is~odd;\\
\frac{n^2+4n-4}{4},&\ n~is~even.
\end{cases}
\end{equation}
\end{lemma}
\begin{proof}Suppose $D\in \{D_1,D_2,\ldots,D_6\}$ contains $u_1\rightarrow u_2\rightarrow u_4$ and $u_1\rightarrow u_3\rightarrow u_4$ with $u_2\ne u_3$.

For $i\in \{1,\ldots,6\}$, every vertex of $D_i$ has at most one predecessor in $\V_1$. Hence $u_2$ and $u_3$ can not both belong to $\V_1$. On the other hand, since each vertex has at most one successor in $\V_2$, $u_2$ and $u_3$ can not both belong to $\V_2$. Without loss of generality, we suppose $u_2\in \V_1$ and $u_3\in \V_2$. We have the following 4 cases.

{\it Case 1.} $u_1,u_4\in \V_1$. In $D_1,D_3,D_5$, each $2$-walk in $D(\V_1)$ originates at a vertex who has no successors in $\V_2$. In $D_4$, there exists no 2-walks in $D(\V_1)$. In $D_2,D_6$, among the vertices in $\V_1$, only $x$ has successors in both $\V_1$ and $\V_2$. The only 2-walk in $D(\V_1)$ with initial vertex $x$ is $x\rightarrow y\rightarrow x$. But the unique successor of $x$ in $\V_2$, namely $w$, is not a predecessor of $x$.

{\it Case 2.} $u_1\in \V_1$ and $u_4\in \V_2$. In $D_1$, $D_3$ and $D_5$, $D(\V_2)$ contains no arcs, and hence there exists no $u_3\rightarrow u_4$. In $D_2$ and $D_4$, among the vertices of $\V_1$ only $x$ has successors in both $\V_1$ and $\V_2$. We know $x$ has a unique successor in $\V_2$, namely $w$, but $w$ has no successor in $\V_2$. Then there exists no 2-walks from $x$ via $\V_2$ to $ \V_2$. In $D_6$, among the vertices of $\V_1$ only $x$ has successors in both $\V_1$ and $\V_2$. We know $x$ has a unique successor in $\V_2$, say $w$, and $w$ has a unique successor $z\in \V_2$. But there exists no 2-walk from $x$ to $z$ via $\V_1$.

{\it Case 3.} $u_1\in \V_2$ and $u_4\in \V_1$. In $D_1$, $D_3$ and $D_5$, $D(\V_2)$ contains no arcs. Then there exists no $u_1\rightarrow u_3$. In $D_2$, $D_4$ and $D_6$, if there exists $u_1$ has a successor $u_3\in \V_2$, then we have $N^+_{\V_1}(u_1)\nrightarrow N^+_{\V_1}(u_3)$.

{\it Case 4.} $u_1,u_4\in \V_2$. In $D_1$, $D_3$, $D_5$, $D(\V_2)$ contains no 2-walks, then $D$ contains no $u_1\rightarrow u_3\rightarrow u_4$. In $D_2$ and $D_4$, all 2-walks in $D(\V_2)$ end at $w$. For all $u\in \V_2$ with a successor in $\V_2$, we have $N^+_{\V_1}(u)=\V_1\backslash \{x\}$. Combining with $\V_1\backslash \{x\}\nrightarrow w$, there exists no $u_1\rightarrow u_2\rightarrow u_4$. In $D_6$, all 2-walks in $D(\V_2)$ end at $w$ or $z$. For all $u\in \V_2$ with successors in $\V_2$, we have $N^+_{\V_1}(u)=\V_1\backslash \{x\}$. Combining with $\V_1\backslash \{x\}\nrightarrow \{w,z\}$, there exists no $u_1\rightarrow u_2\rightarrow u_4$.

In the above cases, none of $D_1,\ldots,D_6$ contains the required 2-walks $u_1\rightarrow u_2\rightarrow u_4$ and $u_1\rightarrow u_3\rightarrow u_4$, a contradiction. Therefore, $D_1,\ldots,D_6$ are $\mathscr{F}$-free. By direct computation, we obtain $e(D_1)=\frac{n^2+4n-4}{4}$ and $e(D_i)=\frac{n^2+4n-5}{4}$ for $i=2,\dots,6$. Hence, we obtain (\ref{1.1}).

\end{proof}
Let $D=(\V,\A)$ be a digraph. For a fixed vertex $v\in \mathcal{V}$, denote by
\begin{equation*}\V_1(v)=N^+(v) ~{\rm and}~ \V_2(v)=\mathcal{V}\backslash \V_1(v),
 \end{equation*}\begin{equation*}
\V_3(v)=\{u\in \V_2(v)|N^+(u)=\V_1(v)\}~ {\rm and} ~\V_4(v)=\V_2(v)\backslash \V_3(v).
 \end{equation*}
 The index $v$ will be omitted if no confusion from the context.

 The following lemma is obvious.

	\begin{lemma}\label{le2}
			Let $D=(\V,\A)$ be an $\mathscr{F}$-free digraph. Then
		\begin{itemize}
			\item[(i)]   two distinct successors of a vertex share no common successor;
			\item[(ii)]given any $v\in \V$,  $e(\V_1(v),u)\le 1$ for all $u\in \V$.
		\end{itemize}
	\end{lemma}

\begin{lemma}\label{le3}
Let $D=(\V,\A)$ be an $\mathscr{F}$-free digraph and $v\in \V$. If there exists a 2-walk $$t_1\rightarrow t_2\rightarrow t_3 ~{\rm in}~ D(\V_1),$$ and $$N^+(u)=\V_1 ~{\rm for ~all} ~u\in\V_2,$$ then $t_1$ has no successor in $\V_2$.
\end{lemma}
\begin{proof}Suppose $t_1$ has a successor $t_4\in \V_2$. Since $N^+(t_4)=\V_1$, we have $t_4\rightarrow t_3$. Hence we obtain two distinct 2-walks from $t_1$ to $t_3$, a contradiction with $D$ being $\mathscr{F}$-free.
\end{proof}
 Let $\Delta^+(D)$ and $\Delta^-(D)$ denote the maximum outdegree and indegree of $D$. If no confusion arises, we write  $\Delta^+$ and $\Delta^-$, respectively.

\begin{lemma}\label{le4}
 Let $D=(\V,\A)\in EX(n)$. If $n$ is odd, then $\Delta^+\in \{\frac{n- 1}{2},\frac{n+ 1}{2}\}$; if $n$ is even, then $\Delta^+\in \{\frac{n}{2}-1,\frac{n}{2},\frac{n}{2}+1\}$.
\end{lemma}
\begin{proof} Let $v\in\V$ such that $d^+(v)=\Delta^+$. We count the size of $D$ in the following way
\begin{eqnarray*}
e(D)=e(\V_1,\mathcal{V})+e(\V_2,\mathcal{V})\le n+(n-\Delta^+)\Delta^+= -(\Delta^+-\frac{n}{2})^2+\frac{n^2+4n}{4}.
\end{eqnarray*}
The inequality is derived from Lemma \ref{le2} and the definition of $\Delta^+$. From (\ref{1.1}), we get the lemma.
\end{proof}

Let $$\alpha=\max\limits_{\substack{v\in \V ,\\ d^+(v)=\Delta^+  }}\max\limits_{u\in \mathcal{V}}{e(u,\V_2(v))}.$$ We give an upper bound on $\alpha$ as follows.
\begin{lemma}\label{le5}
 Let $D=(\V,\A)\in EX(n)$ and $n\ge 8$. Then $\alpha\le 1$.
 \end{lemma}

 \begin{proof} Suppose $\alpha\ge 3$. Then there exist $v$ and $u$ such that $d^+(v)=\Delta^+$ and $u$ has 3 successors $u_1,u_2,u_3\in \V_2(v)$. By Lemma \ref{le2}, we have $\sum\limits_{i=1}^{3}d^+(u_i)\le n$ and

\begin{eqnarray*}
e(D)&=& \sum\limits_{w\in \V_2\backslash\{u_1,u_2,u_3\}}d^+(w)+\sum\limits_{i=1}^{3}d^+(u_i)+\sum\limits_{w\in \V}e(\V_1,w)\\
&\le& (n-\Delta^+-3)\Delta^++n+n\\
&=&-(\Delta^+-\frac{n-3}{2})^2+\frac{n^2+2n+9}{4}.
\end{eqnarray*}
From (\ref{1.1}), we obtain $e(D)<ex(n)$, a contradiction. Hence, we have $\alpha\le 2$.

Now suppose $\alpha=2$. Then there exist $v$ and $u$ such that $d^+(v)=\Delta^+$ and $u$ has two successors $u_1,u_2\in \V_2(v)$. We claim that one of $u_1,u_2$ has $\Delta^+$ successors. Otherwise, we have $d^+(u_1)+d^+(u_2)\le 2\Delta^+-2$. Then
 \begin{eqnarray*}
 e(D)&=&\sum\limits_{w\in \V_2\backslash\{u_1,u_2\}}d^+(w)+\sum\limits_{i=1}^{2}d^+(u_i)+\sum\limits_{w\in \V}e(\V_1,w)\\
 &\le& (n-\Delta^+-2)\Delta^++ 2\Delta^+-2+n\\
 &=& (n-\Delta^+)\Delta^+-2+n\\
 &\le&\frac{n^2+4n}{4}-2.
 \end{eqnarray*}
 From (\ref{1.1}) we obtain $e(D)<ex(n)$, a contradiction. Hence we could assume $d^+(u_1)=\Delta^+$.

 By Lemma \ref{le2}, $u_2$ shares no common successor with $u_1$, i.e.,$$e(u_2,\V_1(u_1))=0.$$ By the definition of $\alpha$, we have $$e(u_2,\V_2(u_1))\le \alpha.$$ It follows that $d^+(u_2)\le \alpha= 2$. Hence, by (\ref{1.1}) we obtain
 \begin{eqnarray*}
e(D)&=& \sum\limits_{w\in \V_2\backslash\{u_1,u_2\}}d^+(w)+\sum\limits_{i=1}^{2}d^+(u_i)+\sum\limits_{w\in \V}e(\V_1,w)\\
&\le& (n-\Delta^+-2)\Delta^++\Delta^++2+n\\
&=&-(\Delta^+-\frac{n-1}{2})^2+\frac{n^2+2n+9}{4}\\
&<&ex(n).
\end{eqnarray*}
Therefore, we have $\alpha\le 1$.
\end{proof}



Let $D=(\V,\A)$ be a digraph. Let $v\in \V$ with $d^+(v)=\Delta^+$ and $u\in \V_2$. If $e(u,\V_1)=\Delta^+-1$, we denote the unique vertex of $\V_1\backslash N^+_{\V_1}(u)$ by $u'$.

\begin{lemma}\label{le6}
Let $D=(\V,\A)\in EX(n)$ with $n\ge 8$ and let $v\in \V$ with $d^+(v)=\Delta^+$. If $u_1\rightarrow u_2$ is in $D(\V_2(v))$, then $$N^+_{\V_1(v)}(u_1)\nrightarrow N^+_{\V_1(v)}(u_2).$$ Moreover, if $\V_1(v)\rightarrow \V_1(v)$ and $d^+(u_1)=d^+(u_2)=\Delta^+$, then $$u_1'\rightarrow N^+_{\V_1(v)}(u_2),~ u_2\in \V_4(v) ~{\rm and} ~N^+_{\V_1(v)}(u_2)=\V_1(v)\backslash \{u_1'\}.$$
\end{lemma}
\begin{proof}  Suppose $N^+_{\V_1(v)}(u_1)\rightarrow N^+_{\V_1(v)}(u_2)$. Then there exists $u_3\in N^+_{\V_1(v)}(u_1)$ and $u_4\in N^+_{\V_1(v)}(u_2)$ such that $u_3\rightarrow u_4$. We have $u_1\rightarrow u_3\rightarrow u_4$ and $u_1\rightarrow u_2\rightarrow u_4$, a contradiction.

For the second part, since $u_1\rightarrow u_2$ is in $D(\V_2(v))$, by Lemma \ref{le5} we have $e(u_1,\V_2(v))= 1$. It follows that $$e(u_1,\V_1(v))=\Delta^+-1{\rm ~ and~} N^+_{\V_1(v)}(u_1)=\V_1\backslash \{u_1'\}.$$
Since $N^+_{\V_1(v)}(u_1)\nrightarrow N^+_{\V_1(v)}(u_2)$ and $\V_1(v)\rightarrow \V_1(v)$, we have $u_1'\rightarrow N^+_{\V_1(v)}(u_2)$. It follows that $u_1'\notin N^+(u_2)$ and $u_2\in \V_4(v)$. By Lemma \ref{le5}, we have $e(u_2,\V_2(v))\le 1$ and $e(u_2,\V_1(v))\ge \Delta^+-1$. It follows that $N^+_{\V_1(v)}(u_2)=\V_1\backslash \{u_1'\}$.
\end{proof}
\begin{lemma}\label{le9}Let $D=(\V,\A)\in EX(n)$ with $n\ge 8$ and let $v\in\V$ with $d^+(v)=\Delta^+$. If $\V_1\rightarrow \V$ and $$d^+(u)=\Delta^+~{\rm ~for ~all ~}u\in \V_2,$$ then $$N^+(u)=\V_1~{\rm ~for ~all ~}u\in \V_2.$$
\end{lemma}
\begin{proof} Suppose there exists an arc $(u_1,u_2)$ in $D(\V_2)$. By Lemma \ref{le6}, we have $u_1'\rightarrow \V_1\backslash \{u_1'\}$ and $u_2\in \V_4$. Hence $u_2$ has a successor $u_3\in \V_2$. Since $u_1\rightarrow u_2\rightarrow u_3$ and $u_1\rightarrow \V_1\backslash \{u_1'\}$, we get $\V_1\backslash \{u_1'\}\nrightarrow u_3$. Since $\V_1\rightarrow \V$, we obtain $u_1'\rightarrow u_3$. By Lemma \ref{le6}, we have $u_3\in \V_4$. Then it has a successor $u_4\in \V_2$, which implies $u_1'\rightarrow u_3\rightarrow u_4$. By Lemma \ref{le5}, $u_1'\nrightarrow u_4$. From $\V_1\rightarrow \V$ we have $u_1'\rightarrow \V_1\backslash \{u_1'\}\rightarrow u_4$, which contradicts $D\in EX(n)$.
\end{proof}

\begin{lemma}\label{le7}
Let $D=(\V,\A)\in EX(n)$ with $n\ge 8$. If $\Delta^+\ge \Delta^-$, then $\Delta^+=\lceil\frac{n+1}{2}\rceil$.
\end{lemma}

\begin{proof} Let $v\in \V$ such that $d^+(v)=\Delta^+$. Since $\Delta^+\ge \Delta^-$, then
\begin{equation}\label{equ1}
e(\V,\V_1)=\sum\limits_{u\in \V_1}d^-(u)\le \Delta^+\Delta^-\le (\Delta^+)^2.
\end{equation}
Applying Lemma \ref{le5} we obtain
\begin{equation}\label{equ2}
e(\V,\V_2)=\sum\limits_{u\in\V\backslash \{v\}}e(u,\V_2)\le n-1.
\end{equation} It follows that
\begin{equation}\label{equ3}
e(D)=e(\V,\V_1)+e(\V,\V_2)
\le(\Delta^+)^2+n-1.
\end{equation} From (\ref{1.1}) we obtain $e(D)<ex(n)$ when $\Delta^+<\frac{n}{2}$. Hence, $\Delta^+\ge \frac{n}{2}$. Combining with Lemma \ref{le4}, we see that $\Delta^+=\frac{n+1}{2}$ when $n$ is odd and $\Delta^+\in \{\frac{n}{2},\frac{n}{2}+1\}$ when $n$ is even.

Now we consider the case $n$ is even and $\Delta^+=\frac{n}{2}$. By (\ref{1.1}), (\ref{equ3}) and $D\in EX(n)$, we have $e(D)=\frac{n^2+4n-4}{4}$. Combining with (\ref{equ1}), (\ref{equ2}) and (\ref{equ3}), we obtain
\begin{equation}\label{eq1}
d^-(u)=\Delta^+~{\rm for~all}~u\in \V_1.
\end{equation}and
$$e(\V,\V_1)=\frac{n^2}{4},~e(\V,\V_2)=n-1.$$
By Lemma \ref{le5}, each vertex in $\V\backslash \{v\}$ has a unique successor in $\V_2$.

By Lemma \ref{le2}, we have $e(\V_1,\V)\le n$. It follows from (\ref{1.1}) that
$$e(\V_2,\V)=e(D)-e(\V_1,\V)\ge (n-\Delta^+)\Delta^+-1,$$ which implies there exists at least $\Delta^+-1$ vertices in $\V_2$ with outdegree $\Delta^+$.

Let $u_1\in \V_2\backslash \{v\}$ such that $d^+(u_1)=\Delta^+$ and $u_1$ has a successor $u_2\in \V_2$. We assert that either $u_2\nrightarrow u_1'$ or $u_1'$ has no predecessor in $\V_1$. Otherwise we have $u_2\rightarrow u_1'$ and there exists $u_1^*\in \V_1\backslash \{u_1'\}$ such that $u_1^*\rightarrow u_1'$. Then there exist two distinct 2-walks $u_1\rightarrow u_2\rightarrow u_1'$ and $u_1\rightarrow u_1^*\rightarrow u_1'$, a contradiction.
By Lemma \ref{le2}, we obtain $e(\V_1,u_1')\le 1$. It follows that
 $$d^-(u_1')=e(\V_2\backslash \{u_1,u_2\},u_1')+e(\{u_1,u_2\},u_1')+e(\V_1,u_1')\le \Delta^+-1,$$ which contradicts (\ref{eq1}). Hence, we have $\Delta^+=\frac{n}{2}+1$.

\end{proof}

\section{Proof of Theorem \ref{th1}}
Now we are ready to present the proof of Theorem \ref{th1}.
\begin{proof} Let $D=(\V,\A)\in EX(n)$. Note that $D\in EX(n)$ if and only if $D'\in EX(n)$. Without loss of generality, we assume $\Delta^+\ge \Delta^-$. Let $v\in \V$ such that $d^+(v)=\Delta^+$.  Keep in mind that given $u\in \V_2$ with outdegree $\Delta^+$, $u\in \V_4$ if and only if $u$ has exactly one successor in $\V_2$.

 By Lemma \ref{le2}, we have $$e(\V_1)\le |\V_1|~{\rm and ~}e(\V_1,\V_2)\le |\V_2|.$$
It follows that \begin{eqnarray}\label{equ3.2}
e(D)= e(\V_1)+e(\V_1,\V_2)+e(\V_2,\mathcal{V})\le |\V_1|+|\V_2|+|\V_2|\Delta^+=n+(n-\Delta^+)\Delta^+.
\end{eqnarray}We distinguish two cases according to the parity of $n$.

(1) $n$ is even. By Lemma \ref{le7}, we have $\Delta^+=\frac{n}{2}+1$. Then from (\ref{equ3.2}) we obtain
\begin{eqnarray*}
e(D)\le \frac{n^2+4n-4}{4}.
\end{eqnarray*}
It follows from (\ref{1.1}) that
\begin{equation}\label{equ4}
e(D)=ex(n)=\frac{n^2+4n-4}{4}.
\end{equation}
Now (\ref{equ3.2}) and (\ref{equ4}) lead to
\begin{equation*}
d^+(u)=\Delta^+~{\rm for~ all}~ u\in \V_2
\end{equation*}
 and
 $$e(\V_1)=|\V_1|~{\rm and}~e(\V_1,\V_2)=|\V_2|,$$ which implies
 \begin{equation}\label{e1}
\V_1\rightarrow \V_1 ~{\rm and}~\V_1\rightarrow \V_2.
\end{equation}
By Lemma \ref{le9}, we have
\begin{equation}\label{e3}
N^+(u)=\V_1 ~{\rm for~ all}~ u\in \V_2.
\end{equation}
Since $\alpha\le 1$, there exist $\frac{n}{2}-1$ vertices in $\V_1$ with exactly one successor in $\V_2$, leaving two vertices $y_1$ and $y_2$ in $\V_1$ with no successors in $\V_2$. Moreover,
  \begin{equation}\label{e2}
\V_1\backslash \{y_1,y_2\} ~{\rm matches}~ \V_2.
\end{equation}
\indent If there exists $z\in \V_1\backslash\{y_1,y_2\}$ such that $z\rightarrow y_2$, we know $z$ has a predecessor in $\V_1$. By Lemma \ref{le3}, we have $y_1\rightarrow z$ or $y_2\rightarrow z$. If the latter one holds, we obtain $z\rightarrow y_2\rightarrow z$, a contradiction with Lemma \ref{le3}. If the former one holds, we consider the predecessor of $y_1$ in $\V_1$. By Lemma \ref{le3}, we have $y_2\rightarrow y_1$. Thus, we obtain $z\rightarrow y_2\rightarrow y_1$, a contradiction with Lemma \ref{le3}. Hence, we have $y_1\rightarrow y_2$. Similarly, $y_2\rightarrow y_1$.

By (\ref{e1}) every vertex $u\in \V_1\backslash \{y_1,y_2\}$ has a unique predecessor $u_1\in \V_1$. By Lemma \ref{le3}, either $u\in N^+(y_1)\cup N^+(y_2)$ or $u_1\in N^+(y_1)\cup N^+(y_2)$. By the first part of (\ref{e1}), we have $D(\V_1)=T(y_1,y_2)$. Combining with (\ref{e3}) and (\ref{e2}), $D$ is an isomorphism of $D_1$.

(2) For $n$ is odd. By Lemma \ref{le7}, we have $\Delta^+=\frac{n+1}{2}$. It follows from (\ref{equ3.2}) that
\begin{equation}\label{equ5}ex(n)\le \frac{n^2+4n-1}{4}.
 \end{equation}Suppose equality in (\ref{equ5}) holds. From (\ref{equ3.2}), we obtain $d^+(u)=\Delta^+$ for all $u\in \V_2$ and $\V_1\rightarrow \V$. By Lemma \ref{le9}, (\ref{e3}) holds. Since $|\V_1|=|\V_2|+1$ and $\alpha\le 1$, there exists $y\in \V_1$ with no successor in $\V_2$ and
\begin{equation}\label{e4}
\V_1\backslash \{y\} ~{\rm matches}~ \V_2.
\end{equation}

Since $\V_1\rightarrow \V_1$, there must exist a cycle whose length is larger than or equal to 2. It follows that there exists a walk $t_1\rightarrow t_2\rightarrow t_3$ with $t_1\ne y$. By (\ref{e4}) $t_1$ has a successor $t_4\in \V_2$, which contradicts Lemma \ref{le3}.

Therfore, $ex(n)\le \frac{n^2+4n-5}{4}$. From (\ref{1.1}), we obtain $$e(D)=ex(n)=\frac{n^2+4n-5}{4}.$$ By Lemma \ref{le2} we have $e(\V_1,\V)\le n$, which implies that $e(\V_2,\V)\ge \frac{n^2-5}{4}$. Hence, at least $\frac{n-3}{2}$ vertices of $\V_2$ have outdegrees $\Delta^+$ and $d^+(u)\ge \frac{n-1}{2}$ for all $u\in \V_2$. Since $\alpha\le 1$, we have $$e(u,\V_1)\ge \frac{n-3}{2}~{\rm for~ all}~ u\in \V_2.$$
We will use the following claim repeatedly.

 {\it{\bf Claim 1.} Let $u\in \V_1$ with $e(u,\V_1)\ge 3$. If it has a predecessor $u^*$, then $u^*$ has no successors in $\V_2$.
 }

Otherwise, $u^*$ has a successor $u_1\in \V_2$. Combining with $e(u_1,\V_1)\ge \frac{n-3}{2}$ and $e(u,\V_1)\ge 3$, $N^+_{\V_1}(u)\cap N^+_{\V_1}(u_1)$ is not empty, which contradicts Lemma \ref{le2}.

As shown in (\ref{equ3.2}), the size of $D$ is the sum of three parts: $e(\V_1)$, $e(\V_1,\V_2)$, $e(\V_2,\mathcal{V})$. We have proven that the maximum size of $D$ is less by 1 than the sum of the maximum numbers of these three parts. Hence, we can deduce that two of three parts achieve the maximum numbers and one of them misses one from the maximum number. We present these three maximum numbers as follows.
\begin{equation}\label{e13}
e(\V_1,\V_1)=\Delta^+,
\end{equation}
\begin{equation}\label{e14}
e(\V_1,\V_2)=n-\Delta^+,
\end{equation}
and
\begin{equation}\label{e15}
 d^+(u)=\Delta^+ {\rm ~for~ all~} u\in\V_2.
\end{equation}
Moreover, (\ref{e13}) is is equivalent to
\begin{equation}\label{e16}
\V_1\rightarrow \V_1;
\end{equation}
and (\ref{e14}) is is equivalent to
\begin{equation}\label{e17}
\V_1\rightarrow \V_2.
\end{equation}
According to the above analysis, we distinguish three cases:

{\it Case 1.} (\ref{e13}) and (\ref{e14}) holds. Then $\frac{n-3}{2}$ vertices of $\V_2$ have outdegrees $\frac{n+1}{2}$ and a vertex $w\in \V_2$ has outdegree $\frac{n-1}{2}$. By Lemma \ref{le5}, we have $\alpha\le 1$. Then there exist $|\V_2|$ vertices of $\V_1$ having exactly one successor in $\V_2$, leaving only one vertex $y\in \V_1$ having no successor in $\V_2$. It follows that
\begin{equation}\label{e22}
\V_1\backslash \{y\} {\rm ~matches~} \V_2.
\end{equation}

Now we distinguish three cases.

{\it Subcase 1.1.} $D(\V_2)$ contains no arc. Then
\begin{equation}\label{e23}
N^+(u)=\V_1~{\rm for~all}~  u\in \V_2\backslash \{w\}~{\rm and~} N^+(w)\subset \V_1.
 \end{equation}
\indent We assert that if there exists $t_1\rightarrow t_2\rightarrow t_3$ in $D(\V_1)$, then either $t_1=y$ or $t_1\rightarrow w$. Otherwise, there exists $t\in \V_2\backslash \{w\}$ such that $t_1\rightarrow t$. From (\ref{e23}), we obtain a $2$-walk $t_1\rightarrow t\rightarrow t_3$, a contradiction.

Since $\V_1\rightarrow \V_1$, $D(\V_1)$ must contain a cycle. From the above assertion, we obtain $D(\V_1)$ contains only one cycle $z\leftrightarrow y$, where $z\in \V_1$ is the predecessor of $w$. Combining with $z\rightarrow w\rightarrow \V_1\backslash \{w'\}$, we obtain $z\notin \V_1\backslash \{w'\}$. Otherwise we have $z\rightarrow w\rightarrow z$ and $z\rightarrow y\rightarrow z$, a contradiction. Thus we have $z=w'$, which implies $w'\rightarrow w$. Moreover, $y$ has only one successor in $\V_1$, i.e.,
 \begin{equation}\label{e25}
N^+_{\V_1}(y)=w'
\end{equation}
\indent We assert
\begin{equation}\label{e24}N^+_{\V_1}(w')= \V_1\backslash \{w'\}.
 \end{equation}Otherwise, there exists $t\in \V_1\backslash \{w'\}$ with $w'\nrightarrow t$. We know $t$ has a predecessor $u\in \V_1$. Since $2$-walks in $D(\V_1)$ have to emanate from $y$ or $w'$, from (\ref{e25}) we have $w'\rightarrow u\rightarrow t$. At the same time, we have $w'\rightarrow w\rightarrow t$ since $N^+(w)=\V_1\backslash \{w'\}$, a contradiction. Hence, we obtain $w'\rightarrow \V_1\backslash \{w'\}$. It follows that
 \begin{equation}\label{e26}
 D(\V_1)=S_{y}(w').
\end{equation}
\indent Since (\ref{e22}) and $w'\rightarrow w$, we have $\V_1\backslash \{y,w'\}$ matches $\V_2\backslash \{w\}$. Combining with (\ref{e23}), (\ref{e24}) and (\ref{e26}), $D$ is an isomorphism of $D_2$ with $T'(w)$ vanishing.

{\it Subcase 1.2.} There exists an arc in $D(\V_2\backslash \{w\})$. Suppose $u_1\rightarrow u_2$ with $u_1,u_2\in \V_2\backslash \{w\}$. By Lemma \ref{le6}, we obtain $u_1'=u_2'$ and $u_1'\rightarrow \V_1\backslash \{u_1'\}$. Since $\V_1\rightarrow \V_1$, $u_1'$ has a predecessor $t\in \V_1$. By Claim 1 we have $y\rightarrow u_1'$. It follows that $$D(\V_1)=S_{y}(u_1').$$ Moreover, $u_1'$ has a unique successor $s\in \V_2$.

If $s\in \V_3$, we have $u_1'\rightarrow s\rightarrow u_1'$, a contradiction with $u_1'\rightarrow y\rightarrow u_1'$. If $s$ has a successor $s_1\in \V_2$, we have $u_1'\rightarrow s\rightarrow s_1$. On the other hand, from (\ref{e22}) we obtain $u_1'\rightarrow \V_1\backslash \{u_1'\}\rightarrow s_1$, a contradiction. Hence, we have
\begin{equation}\label{e27}
s=w ~{\rm and}~ N^+(w)\subset \V_1.
 \end{equation}
\indent We assert that
 \begin{equation*}\label{e28}
u_1'=w'.
 \end{equation*} Otherwise, we have $w\rightarrow u_1'\rightarrow \V_1\backslash \{u_1'\}$. On the other hand, since $w\rightarrow \V_1\backslash \{w'\}$, then $\V_1\backslash \{w'\}\nrightarrow \V_1\backslash \{u_1'\}$. Combining with (\ref{e16}), we obtain $w'\rightarrow \V_1\backslash \{u_1'\}$. Recalling $u_1'\rightarrow \V_1\backslash \{u_1'\}$, we obtain a contradiction with Lemma \ref{le2}. Then we have
 \begin{equation}\label{e29}D(\V_1)=S_{y}(w').
\end{equation}
\indent We assert that
\begin{equation}\label{e30}N^+_{\V_1}(u)=\V_1\backslash \{w'\}~ {\rm for ~all}~ u\in \V_4.
\end{equation}Otherwise, there exists a vertex $u\in \V_4\backslash \{w\}$ such that $u\rightarrow w'$. By Claim 1, we have $u$ has no successor in $\V_2$, which contradicts $u\in \V_4\backslash \{w\}$.

We assert that if there exists a 2-walk $u_3\rightarrow u_4\rightarrow u_5$ in $D(\V_2)$, then $u_5=w$. Otherwise, $u_5\ne w$. We have $u_3\ne w$ from (\ref{e27}). By (\ref{e22}) and $u_1'\rightarrow w$, we obtain $u_3\rightarrow \V_1\backslash \{u_1'\}\rightarrow u_5$, a contradiction. Therefore, for any $(t_1,t_2)$ in $D(\V_2)$, we have either $t_2=w$ or $t_2\in N^-(w)$. It follows that $D(\V_2)=T'(w)+aC_1$ with $a\ge 1$. Combining with (\ref{e22}), (\ref{e27}), (\ref{e29}) and (\ref{e30}), $D$ is an isomorphism of $D_2$.

Notice that there exists $(u_1, u_2)$ in $D(\V_2\backslash \{w\})$. There must exist a 2-walk in $D(\V_2)$.  Thus, in this case, $T'(w)$ must contain a 2-walk.

{\it Subcase 1.3.} $D(\V_2\backslash \{w\})$ contains no arc but $D(\V_2)$ contains arcs. We distinguish two cases.

{\it Subcase 1.3.1.} $N^+(w)\subset \V_1$. Then $\V_4\backslash \{w\}$ is not empty and
\begin{equation*}\label{e31}u\rightarrow w ~{\rm for~ all}~ u\in \V_4\backslash \{w\},
 \end{equation*}which leads to
 \begin{equation}\label{e34}
D(\V_2)=S'(w)+aC_1,
\end{equation}
where $a\ge 1$. By Lemma \ref{le6}, for any $u\in \V_4\backslash \{w\}$ we have $\V_1\backslash \{u'\}\nrightarrow \V_1\backslash \{w'\}$. Since $\V_1\rightarrow \V_1$, we have $u'\rightarrow \V_1\backslash \{w'\}$, which implies that $u'\notin \V_1\backslash \{w'\}$. It follows that
 \begin{equation}\label{eq3}u'=w' ~{\rm for~all}~u\in \V_4.
  \end{equation}Moreover, $w'\rightarrow \V_1\backslash \{w'\}$. From (\ref{e16}) $w'$ has a predecessor in $\V_1$. By Claim 1, we have $y\rightarrow w'$. Hence, we obtain
 \begin{equation}\label{e32}D(\V_1)=S_{y}(w').
 \end{equation}
\indent Let $w_1$ be the successor of $w'$ in $\V_2$. If $w_1$ has a successor $w_2\in \V_2$. Then by (\ref{e22}) we have $w'\rightarrow w_1\rightarrow w_2$ and $w'\rightarrow \V_1\backslash \{w'\}\rightarrow w_2$, a contradiction. If $w_1\in \V_3$, then $w_1\rightarrow w'$. Hence, we obtain $w'\rightarrow w_1\rightarrow w'$ and $w'\rightarrow y\rightarrow w'$, a contradiction. Thus, we get $w_1=w$, i.e., $w$ is the successor of $w'$ in $\V_2$. It follows from (\ref{e22})that
\begin{equation}\label{e33}\V_1\backslash \{y,w'\} ~{\rm matches}~ \V_2\backslash \{w\} ~{\rm and}~ w'\rightarrow w.
\end{equation} Combining with (\ref{e34}), (\ref{eq3}), (\ref{e32}) and (\ref{e33}), $D$ is an isomorphism of $D_{2}$, where $T'(w)$ is $S'(w)$. Note that we could consider $S'(w)$ as a special case of $T'(w)$.

{\it Subcase 1.3.2.} $w$ has a successor $w_0\in \V_2$. Let $\{w_1,w_2\}=\V_1\backslash N^+_{\V_1}(w)$. Suppose $N^+(w_0)=\V_1$. By Lemma \ref{le6}, we have $N^+_{\V_1}(w)\nrightarrow \V_1$. From (\ref{e16}), we have \begin{equation}\label{eq4}\{w_1,w_2\}\rightarrow \V_1.
 \end{equation}Since $D$ is loopless, we get
 \begin{equation}\label{eq5}w_1\leftrightarrow w_2.
  \end{equation}By (\ref{e22}) at least one of $\{w_1,w_2\}$ has a successor in $\V_2$. Without loss of generality, we assume $w_1$ has a successor $w_3\in \V_2$. If $N^+(w_3)=\V_1$, we have $w_1\rightarrow w_2\rightarrow w_1$ and $w_1\rightarrow w_3\rightarrow w_1$, a contradiction. Thus, $w_3\in \V_4$.

If $w_3\ne w$, since $w_1\rightarrow w_2\rightarrow N^+_{\V_1}(w_2)$ and $w_1\rightarrow w_3\rightarrow \V_1\backslash \{w_3'\}$, we have $N^+_{\V_1}(w_2)=w_3'$. From (\ref{eq4}) and (\ref{eq5}), we obtain
\begin{equation}\label{e35}
w_1=w_3'~{\rm and}~ D(\V_1)=S_{w_2}(w_1).
 \end{equation}
Since $w_3\in \V_4\backslash \{w\}$, $w_3$ has a successor $w_4\in \V_2$. Hence, from (\ref{e22}) and (\ref{e35}) we have $w_1\rightarrow w_3\rightarrow w_4$ and $w_1\rightarrow \V_1\backslash \{w_1\}\rightarrow w_4$, a contradiction. Therefore, we obtain $w_3=w$, i.e., \begin{equation}\label{e36}
w_1\rightarrow w.
\end{equation}
\indent Since $w_1\rightarrow w\rightarrow \V_1\backslash \{w_1,w_2\}$ and $w_1\rightarrow w_2$, we obtain $w_2\nrightarrow \V_1\backslash \{w_1,w_2\}$. Combining with $\{w_1,w_2\}\rightarrow \V_1$ and $w_2\nrightarrow w_2$, we have $w_1\rightarrow \V_1\backslash \{w_1\}$. It follows that
\begin{equation}\label{e37}D(\V_1)=S_{w_2}(w_1).
 \end{equation}
By (\ref{e22}) and (\ref{e36}), we have $\V_1\backslash \{w_1\}\rightarrow w_0$, which implies $w_1\rightarrow \V_1\backslash \{w_1\}\rightarrow w_0$. On the other hand, we have $w_1\rightarrow w\rightarrow w_0$, a contradiction.

Hence, $N^+(w_0)\ne \V_1$ and $w_0$ has a successor in $\V_2$. If this vertex is not $w$, then $D(\V_2\backslash \{w\})$ contains an arc, which contradicts the given condition of Subcase 3.3. Hence, we have
\begin{equation}\label{e38}w_0\rightarrow w.
 \end{equation}
 By Lemma \ref{le6}, we have $\V_1\backslash \{w_0'\}\nrightarrow \V_1\backslash\{w_1,w_2\}$. Since $\V_1\rightarrow \V_1$, we have $w_0'\rightarrow \V_1\backslash\{w_1,w_2\}$. Since $D$ is loopless, we have $w_0'\in \{w_1,w_2\}$. Without loss of generality, we suppose $w_0'=w_1$. It follows that
 \begin{equation}\label{eq6}w_0\rightarrow \V_1\backslash \{w_1\}.
 \end{equation} and
 \begin{equation}\label{e39}
 w_1\rightarrow \V_1\backslash\{w_1,w_2\}.
  \end{equation}By $w_0\rightarrow w\rightarrow w_0$ and(\ref{eq6}), we have $\V_1\backslash \{w_1\}\nrightarrow w_0$. Combining with (\ref{e22}), we obtain $$w_1\rightarrow w_0.$$ Similarly, since $w\rightarrow w_0\rightarrow w$ and $w\rightarrow \V_1\backslash \{w_1,w_2\}$, we have $\{w_1,w_2\}\rightarrow w.$ Hence, \begin{equation*}w_2\rightarrow w.\end{equation*}
\indent Applying Lemma \ref{le6} on $w\rightarrow w_0$, we have$$\V_1\backslash \{w_1,w_2\}\nrightarrow \V_1\backslash \{w_1\}.$$ Hence, by (\ref{e39}) and (\ref{e22}), we have $w_1\rightarrow \V_1\backslash \{w_1\}$. Now we have $w_1\rightarrow w_2\rightarrow w$ and $w_1\rightarrow w_0\rightarrow w$, a contradiction.

{\it Case 2.} (\ref{e14}) and (\ref{e15}) hold. Then
\begin{equation}\label{e6}
e(\V_1)=\Delta^+-1.
\end{equation}
(\ref{e6}) implies that there exists a vertex $x\in \V_1$, which has no predecessor in $\V_1$, such that
\begin{equation}\label{eq2}\V_1\rightarrow \V_1\backslash \{x\}
 \end{equation}
Note that $|\V_1|=|\V_2|+1$. By Lemma \ref{le5}, there exists a unique vertex $y\in \V_1$ such that $e(y,\V_2)=0$ and
\begin{equation}\label{e7}\V_1\backslash \{y\}~{\rm matches}~ \V_2.
\end{equation}

Now we distinguish three cases.

{\it Subcase 2.1.} $\V_4$ is empty, which implies that
\begin{equation}\label{e8}N^+(u)=\V_1 ~{\rm ~for~ all~} u\in \V_2.
\end{equation}
Suppose $y\ne x$. Then $y$ has a predecessor in $\V_1$. By Lemma \ref{le3}, $y$ has no successor in $\V_1$. It follows that there exists no 2-walks in $D(\V_1)$. Recalling (\ref{eq2}), among the vertices of $\V_1$ only $x$ has successors in $\V_1$, which implies $D(\V_1)=S(x)$ where $x\in \V_1\backslash \{y\}$. Moreover, by (\ref{e7}) $x$ has a successor in $\V_2$. Combining with (\ref{e7}) and (\ref{e8}), $D$ is an isomorphism of a special case of $D_4$, where $\V_3=\V_2$.

Suppose $y=x$. By Lemma \ref{le3}, all 2-walks originate at $y$. For any $(u_1,u_2)$ in $D(\V_1)$, either $u_1=y$ or $u_1\in N^+(y)$. From (\ref{eq2}) we have $D(\V_1)=T(y)$. Combing with (\ref{e7}) and (\ref{e8}), $D$ is an isomorphism of $D_3$.

{\it Subcase 2.2.} There exists an arc $(u_1, u_2)$ in $D(\V_2)$ with $N^+(u_2)=\V_1$. By Lemma \ref{le6}, we have $N^+_{\V_1}(u_1)\nrightarrow \V_1$. By Lemma \ref{le5}, we have $e(u_1,\V_2)= 1$ and $e(u_1,\V_1)=\frac{n-1}{2}$. Hence, we have $N^+_{\V_1}(u_1)= \V_1\backslash \{u_1'\}$. By (\ref{eq2}), we obtain $u_1'\rightarrow \V_1\backslash \{x\}$. Since $D$ is loopless, we have $u_1'=x$ and
\begin{equation}\label{e11}D(\V_1)=S(x).
\end{equation}
By Lemma \ref{le5} and (\ref{e15}), we obtain
\begin{equation*}e(u,\V_1)\ge \Delta^+-1 {\rm ~for~ all~} u\in \V_2.
\end{equation*}
\indent We assert that \begin{equation}\label{e9}
 N^+_{\V_1}(u)=\V_1\backslash \{x\} ~{\rm ~for~ all}~ u\in \V_4.
  \end{equation}
  Otherwise, there exists an arc $(t_1, t_2)$ in $D(\V_2)$ with $t_1\rightarrow x$. Then $e(x,\V_1)=\frac{n-1}{2}$ and $e(t_2,\V_1)\ge\frac{n-1}{2}$ contradict Lemma \ref{le2}.

If $x\ne y$, then $x$ has a successor $x_1\in \V_2$. We assert that \begin{equation}\label{e10}
N^+(x_1)=\V_1.
  \end{equation} Otherwise, $x_1$ has a successor $x_2\in \V_2$ and $x\rightarrow x_1\rightarrow x_2$. On the other hand, by (\ref{e7}) and $x\rightarrow x_1$ we have $x\rightarrow \V_1\backslash \{x\}\rightarrow x_2$. We get two 2-walks from $x$ to $x_2$, a contradiction.

We also assert that if there exists $u_3\rightarrow u_4\rightarrow u_5$ in $D(\V_2)$, then $u_5=x_1$. Otherwise, from (\ref{e9}) we have $u_3\rightarrow \V_1\backslash \{x\}$. By (\ref{e7}) and $x\rightarrow x_1$, we have $u_3\rightarrow \V_1\backslash \{x\}\rightarrow u_5$, a contradiction.

It follows that all 2-walks in $D(\V_2)$ share the same terminal vertex $x_1$. Hence for any arc $t_1\rightarrow t_2$ in $D(\V_2)$, if $N^+(t_2)\ne \V_1$, then $x_1$ is the successor of $t_2$ in $\V_2$. Therefore, $$D(\V_2)=T'(x_1)+\sum\limits_{i=1}^kS'(u_i)+aC_1,$$ with $\V_3\ne \V_2$. Combining with (\ref{e7}), (\ref{e11}) and (\ref{e9}), $D$ is an isomorphism of $D_4$ with $\V_3\ne \V_2$.

Now suppose $x=y$. Then $x$ has no successor in $\V_2$. Let us turn back to consider $u_1\rightarrow u_2$ with $u_1,u_2\in \V_2$. From (\ref{e9}) we obtain $N^+_{\V_1}(u_1)=\V_1\backslash \{x\}$. Then $\V_2(u_1)=\{x\}\cup \V_2\backslash \{u_2\}$. Moreover, in $\V_2(u_1)$ there exist $\frac{n-3}{2}$ vertices with outdegree $\Delta^+$ and one vertex with outdegree $\Delta^+-1$. Replace the role of $v$ by $u_1$ and apply the same arguments as in Case 1, $D$ is an isomorphism of $D_2$.

{\it Subcase 2.3.} $\V_4$ is a proper subset of $\V_2$ and we have $u_2\in \V_4$ for every arc $(u_1,u_2)$ in $D(\V_2)$. Suppose $t_1\rightarrow t_2$ with $t_1,t_2\in \V_2$ and $e(t_2,\V_1)=\Delta^+-1$. Then $t_2$ has a successor $t_3\in \V_2$. Similarly, $t_3$ has a successor $t_4\in \V_2$. Since $t_1\rightarrow t_2\rightarrow t_3$ and $t_1\rightarrow \V_1\backslash \{t_1'\}$, then $\V_1\backslash \{t_1'\}\nrightarrow t_3$. Combining with (\ref{e7}), we have $t_1'\rightarrow t_3$. Similarly, $t_2'\rightarrow t_4$. From (\ref{e7}) we obtain $t_1'\ne t_2'$. Since $t_1\rightarrow t_2\rightarrow \V_1\backslash \{t_2'\}$ and $t_1\rightarrow \V_1\backslash \{t_1'\}$, we have $\V_1\backslash \{t_1'\}\nrightarrow \V_1\backslash \{t_2'\}$. From (\ref{eq2}), we obtain $t_1'\rightarrow \V_1\backslash \{x,t_2'\}$, which implies $e(t_1',\V_1)\ge \Delta^+-2$. Similarly, $e(t_2',\V_1)\ge \Delta^+-2$. Hence there exists $t\in \V_1$ such that $t_i'\rightarrow t$ for $i=1,2$, which contradicts Lemma \ref{le2}.

{\it Case 3.} (\ref{e13}) and (\ref{e15}) holds, and
\begin{equation}\label{e12}
 e(\V_1,\V_2)=n-\Delta^+-1.
\end{equation}
By Lemma \ref{le5}, there exist $|\V_2|-1$ vertices of $\V_1$ having exactly one successor in $\V_2$. Applying Lemma \ref{le2}, there exists one vertex $x\in \V_2$ with no predecessor in $\V_1$. Note that $|\V_1|=|\V_2|+1$. There exist $y_1,y_2\in \V_1$ such that $e(y_i,\V_2)=0$ for $i=1,2$ and
\begin{equation}\label{e18}
\V_1\backslash \{y_1,y_2\}{\rm~ matches~} \V_2\backslash \{x\}.
\end{equation}

Now we distinguish two cases.

{\it Subcase 3.1.} $\V_4$ is empty. Then (\ref{e8}) holds.
From (\ref{e16}), there exist cycles in $D(\V_1)$. Assume $t_1\in \V_1$ is in a cycle of $D(\V_1)$. Then there exists $t_1\rightarrow t_2\rightarrow t_3$ in $D(\V_1)$. Lemma \ref{le3} guarantees that only $y_1,y_2$ could be in a cycle, which implies there is only one cycle $y_1\leftrightarrow y_2$ in $D(\V_1)$. (\ref{e16}) implies that each vertex in $\V_1$ has a unique predecessor in $\V_1$. By Lemma \ref{le3}, $\V_1\subset N^+(y_1)\cup N^+(y_2)\cup N^+(N^+(y_1)\cup N^+(y_2))$. Hence, we have $D(\V_1)=T(y_1,y_2)$. Combining with (\ref{e8}) and (\ref{e18}), $D$ is an isomorphism of $D_5$.

{\it Subcase 3.2.} $\V_4$ is a proper subset of $\V_2$. Then there exists an arc $u_1\rightarrow u_2$ in $D(\V_2)$. By Lemma \ref{le6}, we have  $u_1'\rightarrow \V_1\backslash \{u_1'\}$.

We assert that
\begin{equation}\label{e20}u\rightarrow \V_1\backslash \{u_1'\} ~{\rm ~for~ all~} u\in \V_4.
\end{equation} Otherwise, there exists $u\in \V_4$ such that $u'\ne u_1'$. Then $u$ has a successor $u^*\in \V_2$. By Lemma \ref{le6}, we obtain $u'\rightarrow \V_1\backslash \{u'\}$, a contradiction with Lemma \ref{le2}.

We assert if there exists $t_1\rightarrow t_2\rightarrow t_3$ in $D(\V_2)$, then either $u_1'\rightarrow t_3$ or $t_3=x$. Otherwise, $t_3$ has a predecessor $t_4\in \V_1\backslash \{u_1'\}$. Since $t_1\in \V_4$ and (\ref{e20}), we have $t_1\rightarrow t_4\rightarrow t_3$, a contradiction.

Let us turn back to consider $u_1\rightarrow u_2$ with $u_1,u_2\in \V_2$. Now by Lemma \ref{le6}, we have $u_2\in \V_4$, which means $u_2$ has a successor in $\V_2$. Applying Lemma \ref{le6} repeatedly, we see that $D(\V_2)$ must contain a cycle $C_s$. Since $u_1'$ has a unique successor $z\in \V_2$, according to the above assertion $D(\V_2)$ contains only one 2-cycle $z\leftrightarrow x$. Moreover, for any arc $(t_1,t_2)$ in $D(\V_2)$, either $t_1\in N^-(x)\cup N^-(z)$ or $t_2\in N^-(x)\cup N^-(z)$. By Lemma \ref{le5}, we have
\begin{equation}\label{e21}
D(\V_2)=T'(x,z)+aC_1,
\end{equation}where $a\ge 1$.

From (\ref{e16}), $u_1'$ has a unique predecessor $u_1^*\in \V_1$. By Claim 1 we have $u_1^*\in \{y_1,y_2\}$. Without loss of generality, we let $y_1\rightarrow u_1'$. Note that $x\in \V_4$ as $x\rightarrow z$. From (\ref{e20}) we get $x'=u_1'$. Hence, $D(\V_1)=S_{y_1}(x')$. Combing with (\ref{e18}), (\ref{e20}) and (\ref{e21}), $D$ is an isomorphism of $D_6$.

So far we have proved that if $e(D)=ex(n)$, then $D$ is an isomorphism of $D_i$ with $i\in \{1,\ldots,6\}$. On the other hand, by Lemma \ref{le8}, $D_i$ is $ \mathscr{F}$-free for $i\in \{1,\ldots,6\}$. Hence, we get the second part of Theorem \ref{th1}.

This completes the proof.

\end{proof}
\section*{Acknowledgement}

Partial of this work was done when Huang was visiting Georgia Institute of Technology and Lyu was visiting Auburn University with the financial support of China Scholarship Council. They thank China Scholarship Council,  Georgia Tech and Auburn University for their support. The second author also thanks Professor Tin-Yau Tam for helpful discussions on matrix theory during his visit.

\end{document}